\documentclass[10pt,a4paper]{amsart}
\usepackage[T1]{fontenc}
\usepackage{fullpage} 
\usepackage{amsmath}
\usepackage{amsthm}
\usepackage{amssymb}
\usepackage{mathtools}
\usepackage[only,mapsfrom]{stmaryrd}
\usepackage{graphicx}
\usepackage{cancel} 
\usepackage{color} 
\usepackage{nccmath}
\usepackage[all]{xy} 
\usepackage{tikz}
\usetikzlibrary{arrows,calc,positioning}
\usetikzlibrary{cd}
\usepackage{nicefrac}
\usepackage{enumitem}
\usepackage[colorlinks,citecolor=blue,linkcolor=blue,urlcolor=blue,filecolor=blue,breaklinks]{hyperref}
\usepackage{caption}
\usepackage{comment}
\usepackage{tensor}

\numberwithin{equation}{section}
\numberwithin{figure}{section}

\newtheorem{thm}{Theorem}[section]
\newtheorem{athm}{Theorem}

\newtheorem{lem}[thm]{Lemma}
\newtheorem{prop}[thm]{Proposition}
\newtheorem{cor}[thm]{Corollary}

\newtheorem*{thm*}{Theorem}
\newtheorem*{conj*}{Conjecture}
\newtheorem*{cor*}{Corollary}
\newtheorem*{ques*}{Question}
\newtheorem*{claim*}{Claim}

\theoremstyle{definition}
\newtheorem{rem}[thm]{Remark}
\newtheorem{defn}[thm]{Definition}

\newtheorem{ex}[thm]{Example}

\newtheorem*{rem*}{Remark}

\newcommand{\cC}{\mathcal{C}}

\newcommand{\cP}{\mathcal{P}}

\newcommand{\bQ}{\mathbb{Q}}

\newcommand{\bZ}{\mathbb{Z}}

\newcommand{\RFRS}{\mathrm{RFRS}}

\newcommand{\incl}[3][right]%
{%
\draw[<-,>=#1 hook] #2 to ($ #2!0.5!#3 $);
\draw[->] ($ #2!0.5!#3 $) to #3;%
}
\newcommand{\inclusion}[5][right]%
{%
\draw[<-,>=#1 hook] #4 to ($ #4!0.5!#5 $) node[#2,font=\small]{#3};
\draw[->,>=stealth'] ($ #4!0.5!#5 $) to #5;%
}

\renewcommand{\geq}{\geqslant}
\renewcommand{\leq}{\leqslant}

\title{Virtual inheritance properties of graph products}

\author{Xiaoming Huang}
\address{School of Mathematical Sciences, Fudan University, No. 220 Handan Road, Shanghai, 200433, China}
\email{xmhuang24@m.fudan.edu.cn}

\author{Xiaolei Wu}
\address{Shanghai Center for Mathematical Sciences, Jiangwan Campus, Fudan University, No. 2005 Songhu Road, Shanghai, 200438, P.R. China}
\email{xiaoleiwu@fudan.edu.cn}

\author{Shengkui Ye}
\address{NYU Shanghai, No.567 Yangsi West 
  Rd, Pudong New Area, Shanghai, 200124, P.R. China \\
NYU-ECNU Institute of Mathematical Sciences at NYU Shanghai, 3663 Zhongshan Road North, Shanghai, 200062, China}
\email{sy55@nyu.edu}

\date{May 2026}

\begin{document}

\begin{abstract}
We prove that many virtual properties are closed under taking graph products, including: virtually RFRS, virtually (compact) special, virtually CAT(0) cube, and virtually normally poly-free. Our proof uses Januszkiewicz and   \'Swi\k{a}tkowski's strong commensurability theorem for graph products, for which we provide an elementary proof.

\end{abstract}
\maketitle

\section*{Introduction}

Graph products of groups were first introduced by Green in \cite{Green90} as a common generalization of free products and direct products. Let $\Gamma$ be a simplicial graph with vertex set $V(\Gamma)$ and let $\{G_v\}_{v\in V(\Gamma)}$ be a collection of groups. The graph product $\Gamma G_v$ is obtained from the free product $\ast_{v\in V(\Gamma)}G_v$ by adding the relations that elements of adjacent vertex groups commute. Thus graph products interpolate naturally between free products and direct products. Two particularly important classes of examples are right-angled Artin groups and right-angled Coxeter groups, obtained respectively by taking all vertex groups to be infinite cyclic or cyclic of order two.

Graph products play a central role in geometric group theory. Their subgroup structure, residual properties, automorphism groups, and geometric actions have been studied extensively; see for example \cite{Green90,MeierVanWyk95,HsuWise1999,Charney07,AntolinMinasyan15,genevois2017cubical}. One basic theme in the study of graph products is the investigation of inheritance properties. A large number of group theoretic properties are known to be preserved under taking graph products. For instance, graph products preserve residual finiteness, linearity, soficity, orderability; see  Theorem \ref{thm:graph-prod-sumry} for a summary. 

The main purpose of this paper is to study several  virtual inheritance properties for graph products. Given a group-theoretic property $P$, a group is called \emph{virtually $P$} if it has a finite index subgroup that has $P$. An important tool for studying virtual properties of graph products is Januszkiewicz and   \'Swi\k{a}tkowski's strong commensurability theorem  \cite[Theorem 1]{JanSwi01}. Recall that two groups are said to be \emph{commensurable} if they share a common subgroup of finite index; they are \emph{strongly commensurable} if the indices are the same and \emph{strongly $d$-commensurable} if the common index is $d$. 
Our first observation is the following (see Proposition \ref{lem:common-finite-index-subgroup} for a general form).

\begin{athm} \label{thm:amal-comm}
Let $G_1 = (D_1 \times C) \ast_C B, G_2 = (D_2 \times C) \ast_C B$ be two amalgamated products of groups. Suppose that $C$ is a retract of $B$, and that \(D_1, D_2\) have a common subgroup \(H\) of the same index \(n\). Then $G_1,G_2$
have a common subgroup of index \(n\).    
\end{athm}

As an application, we recover Januszkiewicz and   \'Swi\k{a}tkowski's strong commensurability theorem  \cite[Theorem 1]{JanSwi01}: for two graph products over the same underlying graph, if vertex groups are strongly commensurable, then the two graph products are strongly commensurable (see Corollary \ref{thm:strong-com-gp} for details).  
The original proof of Januszkiewicz and   \'Swi\k{a}tkowski uses right-angled buildings associated to graph products. Our proof is short and elementary using only Bass--Serre theory.   We then use the strong commensurability theorem to prove many properties of groups are closed under taking graph products. Our main result can be stated as the following.

\begin{athm}
\label{thm:intro-virtual}
The following classes of groups are closed under taking graph products:
\begin{enumerate}
    \item (\ref{cor:graph-prod-pass-sub}) Virtually RFRS groups;
    \item (\ref{cor:graph-prod-pass-sub}) Virtually special groups;
    \item (\ref{cor:poly-free} or \ref{cor:graph-prod-pass-sub}) Virtually normally poly-free groups;
    \item (\ref{thm:graph-product}) Virtually compact special groups;
    \item (\ref{thm:graph-product}) Virtually CAT$(0)$ groups;
    \item (\ref{thm:graph-product}) Virtually CAT$(0)$ cube groups;

    \item (\ref{lem:swl}) Groups with positive (resp. discrete) stable word length.
   
\end{enumerate}
\end{athm}
\begin{rem}

\begin{enumerate}

  \item  A sketch of proof of the fact that graph products of RFRS groups are RFRS  based on \cite[Theorem 1.3]{KobSuc2020} was pointed out in \cite[Remark 3.16]{FiHuLe24}. We  give a completely different proof in Theorem \ref{thm:grp-rfrs}.

  \item   Kim proved in \cite{Kim12} that graph products of finite or cyclic groups are virtually compact special. Later Ruan and Witzel \cite{RuanWitzel16} proved that graph products of finitely generated abelian groups are virtually compact special. These are special cases of our result in Case (4). 

  \item  Genevois has proved that graph products of CAT(0) (resp. CAT(0) cube) groups are CAT(0) \cite[Theorem 8.20]{genevois2017cubical} (resp. CAT(0) cube \cite[Theorem 8.17]{genevois2017cubical}). He also has results for groups acting geometrically and specially on CAT(0) cube complexes \cite[Theorem 8.17]{genevois2017cubical}, but this does not seem to work directly for virtually (compact) special groups.

  \item  Leary and Petrosyan recently  proved that graph products of poly-free groups are poly-free \cite[Corollary 1.5]{LeaPet26}. Their proof also implies that graph products of virtually poly-free groups are virtually  poly-free. We shall give a shorter proof of these in \S \ref{section:poly-free}.  

  \item The distinction between virtually special groups and virtually RFRS groups is quite subtle. Essentially, the only currently known examples of groups of finite cohomological dimension that are  RFRS but not virtually special are certain cocompact complex hyperbolic  lattices; see the work of  Agol and Stover \cite{AgolStover23}. These groups cannot act properly on CAT(0) cube complexes \cite{Py13}, hence are not virtually special. For groups of infinite cohomological dimension, there are also examples such as the lamplighter groups $\bZ \wr \bZ$ that are RFRS.  Applying the graph product construction to these groups yields further examples of virtually RFRS groups that are not virtually special. Note also that RFRS groups have VRC, i.e. every cyclic subgroup is a virtual retract \cite[H.18]{AFW15}; see also \cite{Min21}.  Since Agol and Stover's examples are Gromov hyperbolic and RFRS but not virtually special, they are  counterexamples to \cite[Conjecture 13.2]{Wise14}.

    \end{enumerate}
\end{rem}

Let us recall the corresponding notions here briefly. The RFRS property was first introduced by Agol in \cite{Agol08}. A finitely generated group $G$ is called \emph{RFRS} if  there exists a sequence of finite index  subgroups
\[
G=G_0\rhd G_1\rhd \cdots
\]
such that
\[
\bigcap_i G_i=\{1\}
\]
and
\[
G_{i+1}\geq \ker(G_i\to H_1(G_i;\bQ))
\]
for every $i$. The RFRS property plays a fundamental role in $3$-manifold topology and geometric group theory, most notably in Agol's proof of the Virtual Fibering Theorem \cite{Agol13}.

A group is called \emph{CAT(0)} if it acts properly and cocompactly on a CAT(0) space by isometries. It is called \emph{CAT(0) cube} if it  acts properly and cocompactly on a CAT(0) cube complex. It is called \emph{compact special} if it is the fundamental group of a finite A-special cube complex in the sense of Haglund and Wise \cite{Hagwise08}. A group is \emph{special} if it embeds into a RAAG.

A group $G$ is \emph{poly-free} if there exists a finite sequence of subgroups
\[
1=G_0\triangleleft G_1\triangleleft \cdots \triangleleft G_n=G
\]
such that each quotient $G_{i+1}/G_i$ is free. If in addition each $G_i$ is normal in $G$, then $G$ is called \emph{normally poly-free}. (Normally) Poly-free groups appear naturally in the study of braid groups, more generally Artin groups and arrangement groups. In fact,  Bestvina in \cite[Question 2]{Bestvina99} and the discussions below it asks whether all Artin groups are virtually poly-free.



Let \(G\) be a finitely generated group with finite generating set \(S\).
For \(g \in G\), the \emph{word length} of \(g\) with respect to \(S\) is

\[
|g|_S
=
\min \left\{
n \in \mathbb{N}
:
g=s_1s_2\cdots s_n,\;
s_i \in S^{\pm 1}
\right\}.
\]

The \emph{stable word length} (or \emph{translation length}) of \(g\) is defined by

\[
\tau_S(g)
=
\lim_{n\to\infty}\frac{|g^n|_S}{n}.
\]

The limit exists because the sequence \(\{|g^n|_S\}_{n\ge1}\) is subadditive. If $\tau_S(g)>0$ for every infinite-order element $g$, we say the group $G$ has a \emph{positive stable word length}. If there is a uniform positive lower bound for $\{\tau_S(g):g\in G,\tau_S(g)>0\}$, we say the stable word length is \emph{discrete}. It is known that the property of having (discrete) positive stable word length is closed under taking finitely generated subgroups and direct products.  For groups with discrete positive stable word lengths  like mapping class groups, outer automorphism group of free groups, CAT(0) groups, hierarchically hyperbolic groups,  see \cite{Ye24, ABd} and the references therein.

All groups in this paper are assumed to be finitely generated and all graphs are finite unless stated otherwise.

\subsection*{Acknowledgments.}
Wu is currently a member of LMNS. We thank Federico Berlai, Anthony Genevois, Jingyin Huang, and Nansen Petrosyan for helpful communications or comments. We also thank Jingyin Huang for encouraging us to treat the virtually compact special case.

\section{Preliminaries}

Let $\Gamma$ be a finite simplicial graph. We denote the set of vertices and the set of edges by $V(\Gamma)$ and $E(\Gamma)$, respectively. Let $\{ G_v | v\in V(\Gamma)\}$ be a collection of groups. The graph product $\Gamma G_v$ is defined as the free product $\ast_{v\in V(\Gamma)} G_v$, subject to the following relations
$$[ g_v , g_w ]=1 \quad \text{for all } g_v \in G_v, g_w \in G_w \text{ such that } \{ v,w\}\in E(\Gamma). $$

For example, a graph product is a free product if its underlying graph $\Gamma$ has no edge, while it is a direct product if $\Gamma$ is complete.

The theory of graph products was first developed by E. Green in her thesis \cite{Green90}, where she proved 
\begin{thm}
    Every element $g\in \Gamma G_v$ can be represented by a reduced word $g_1 g_2 \cdots g_n (g_i\in G_{v_i})$ in the sense that it cannot be shortened by any finite sequence of the following moves
    \begin{itemize}
        \item Remove a syllable $g_i$ if $g_i=1$;
        \item Replace consecutive syllables $g_i$ and $g_{i+1}$ in the same vertex group $G_v$ with the single syllable $(g_ig_{i+1})$;
        \item For consecutive syllables $g_i \in G_v$ and $g_{i+1} \in G_w$ with $\{v,w\}\in E(\Gamma)$, interchange $g_i$ and $g_{i+1}$.
    \end{itemize}
    Moreover, a reduced word represents the trivial element if and only if it is the empty word.
\end{thm}

Given any full subgraph $\Lambda$ of $\Gamma$, we can form the \emph{full subgroup} $\Lambda G_v \leq \Gamma G_v$, which is also denoted by $G_{\Lambda}$. From the normal form theorem, it follows that $\Lambda G_v$ is exactly the subgroup of $\Gamma G_v$ generated by those $G_v$ such that $v\in \Lambda$. Given $v\in V$, the \emph{link} $lk(v)$ of $v$ is the full subgraph spanned by all vertices adjacent to $v$, and the \emph{star} $st(v)$ of $v$ is the full subgraph spanned by $v$ and $lk(v)$. And we use $\Gamma \setminus \{ v \}$ to denote the full subgraph of $\Gamma$ spanned by vertices in $V(\Gamma) \setminus \{v\}$.

Another useful consequence of the normal form theorem is the following.

\begin{prop}\label{prop:amal-graph}
    For any fixed vertex $v \in V(\Gamma)$, the graph product $G=\Gamma G_v$ splits as an amalgamated free product of full subgroups: $G= G_{st(v)}\ast_{G_{lk(v)}} G_{\Gamma \setminus \{ v \}}=(G_v \times G_{lk(v)})\ast_{G_{lk(v)}} G_{\Gamma \setminus \{ v \}}.$
\end{prop}

\begin{defn}
    Let $\cC$ be a class of groups. A  group $G$ is called residually $\cC$ if for any nontrivial element $g\in G$, there exists a group $C\in \cC$ and a map $\phi: G\to C$ such that $\phi(g)\neq 1$.
\end{defn}


We are interested in the inheritance properties of the graph products. The following theorem summarizes what we have found in the literature. We remind the reader that this is by no means a complete list. 

\begin{thm} \label{thm:graph-prod-sumry}
    The following classes of groups are closed under taking graph products:

    \begin{enumerate}
        \item Groups having solvable word problem (resp. conjugacy problem); 
        \item Residually finite groups;
        \item Residually amenable groups; residually elementary amenable groups;
        \item RFRS groups;
        \item RFR$p$ groups;
        \item Groups that are conjugacy separable;
        
        \item Groups that are hereditarily conjugacy separable;
        \item Groups that do not have pointwise-inner automorphisms;
        \item Groups that are cyclic subgroup separable;
        \item Groups having unique roots;
        \item Groups that are linear over the complex numbers;
        \item Locally indicable groups;
        \item Orderable groups;
        \item Poly-free groups;
        \item Sofic groups;
        \item Groups satisfying the Tits alternative, that is, groups in which every finitely generated subgroup either contains a free subgroup of rank $2$ or is virtually solvable;
        \item (Compact) special groups;
        \item CAT(0) groups;
        \item CAT(0) cube groups;
        \item Groups of type $F_n$ or $FP_n$;
        \item Virtually torsion-free groups;
        \item Virtually poly-free groups;
        \item Virtually orderable groups;
        \item Virtually locally indicable groups;
        \item Groups satisfying VRC;
        \item Groups satisfying the Farrell--Jones Conjecture with finite wreath products and coefficients in additive categories;
        \item Groups satisfying the Baum--Connes Conjecture with coefficients and finite wreath products.
    \end{enumerate}
\end{thm}

\begin{proof}  
\begin{enumerate}

    \item These are proved  in \cite{Green90}.
       
    \item This is first proved in \cite[Theorem 5.4]{Green90}; see also \cite[Theorem 3.7]{HsuWise1999}.
    
    \item \cite[Corollary A]{BerlaiFerov16}.
        \item A sketch of proof was given in \cite[Remark 3.16]{FiHuLe24}. See also Theorem \ref{thm:grp-rfrs} for a different proof.

        \item This follows from \cite[Theorem 1.3]{KobSuc2020}, using the amalgamated product in Proposition  \ref{prop:amal-graph}.

        \item \cite[Theorem 1.1]{Ferov16}.

        \item \cite[Theorem 1.2]{Ferov16}.

        \item \cite[Theorem 1.1]{Ferov16b}.

        \item \cite[Theorem A]{Ferov19}.

        \item \cite[Theorem 6.5]{Ferov19}
        
        \item \cite[Theorem A]{BdlN22}.
        \item \cite[Corollary 5.12]{AntolinMinasyan15}; see also Corollary \ref{cor:poly-free} for a shorter proof. 
        \item \cite[Corollary 5.13]{AntolinMinasyan15}; see also Corollary \ref{cor:poly-free} for a shorter proof. 
        \item \cite[Corollary 1.5]{LeaPet26}; see also Corollary \ref{cor:poly-free} for a shorter proof. 
        \item \cite{CioHolRee14}.
        \item \cite{AntolinMinasyan15}.

          \item This follows immediately by taking polyhedral products of the corresponding (compact) special complexes according to the graph.

        \item \cite[Theorem 8.20]{genevois2017cubical}.
        \item \cite[Theorem 8.18]{genevois2017cubical}.
        \item  \cite[Theorem 2]{Alonso96}, \cite{Cohen95}.
     
     \item  \cite[Corollary 3.4]{JanSwi01}. It also follows from \cite[Corollary 1.5]{LeaPet26}.
     
         \item This  follows from \cite[Corollary 1.5]{LeaPet26}.
         

            \item This follows from \cite[Corollary 1.5]{LeaPet26}. In fact, the graph product $\Gamma G_v$ surjects to $\prod G_v$ with kernel a subgroup of a RAAG. The statement now follows from the facts: RAAGs are orderable; orderable groups are closed under extensions.
            
        \item This also follows from  \cite[Corollary 1.5]{LeaPet26}; see Corollary \ref{cor:graph-prod-pass-sub} for a shorter proof.
        
         \item \cite[Theorem 7.1]{Min21}.
         
         \item See \cite{GanRup13}; See Corollary \ref{cor:poly-free} for a shorter proof.
         \item \cite[Corollary 1.8]{LeaPet26}.
\end{enumerate}
\end{proof}

\section{Strong commensurability theorem for graph products}

In this section, we give an elementary proof of Januszkiewicz and   \'Swi\k{a}tkowski's strong commensurability theorem for graph products. 

Recall that a subgroup $H\leq G$ is called a \emph{retract} if there is a homomorphism $r:G\to H$ such that $r|_H=\mathrm{id}_H$. Theorem \ref{thm:amal-comm} is a special case of the following technical but elementary lemma (with $A_1=D_1 \times C, A_2=D_2 \times C$). 

\begin{prop}\label{lem:common-finite-index-subgroup}

Let $G_i = A_i \ast_C B,i=1,2$ be two amalgamated products of groups. Suppose  that
\begin{enumerate}
    \item $C$ is a retract of $B$, and  $C\lhd A_i$ for both $i$; 
    \item $A_1$ and $A_2$ have a common subgroup $H\geq C$ of the same index $n$;
    \item there are representatives $\{ g_{i1} , g_{i2}, \cdots , g_{in} \}$ of right cosets for $H \backslash A_i, i=1,2,$ such that they lie in the centralizer of $C$, i.e., $g_{ij}cg_{ij}^{-1}=c$ for any $c\in C$, $i=1,2$, and $j=1,2,...,n$.
\end{enumerate}
Then $G_1,G_2$ have isomorphic subgroups of index \(n\). In fact, let $r:B\to C$ be the retraction map and $\pi_i:G_i\to A_i$ be the homomorphism induced by
\[
\pi_i|_{A_i}=\mathrm{id}_{A_i}, \qquad \pi_i|_B=r,
\]
and denote $K_i=\pi_i^{-1}(H)$. Then 
$K_1\cong K_2$.
\end{prop}

\begin{proof}
The proposition follows from the following claim:
 
\textbf{Claim:} Both $K_1$ and $K_2$ are naturally isomorphic to a star-shaped graph of groups with one central vertex group $H$, $n$ outer vertex groups all equal to $B$, and all edge groups equal to $C$, where each edge group embeds into $H$ and $B$ by canonical inclusions. The diagram is drawn below. 


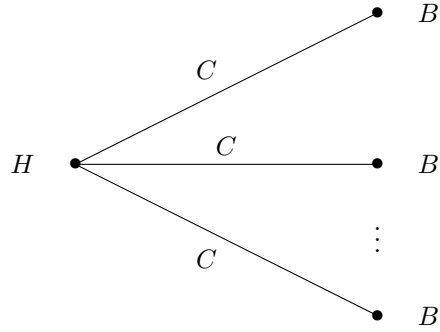
\[
\begin{tikzpicture}[baseline=(current bounding box.center)]
\node (A) at (0,0) {$\bullet$};
\node[left=0.2cm of A] {$H$};

\node (B1) at (4,2) {$\bullet$};
\node[right=0.2cm of B1] {$B$};

\node (Bmid) at (4,0) {$\bullet$};
\node[right=0.2cm of Bmid] {$B$};

\node[below=0.3cm of Bmid] {$\vdots$};

\node (Bn) at (4,-2) {$\bullet$};

\node[right=0.2cm of Bn] {$B$};

\draw (0,0) -- node[above left] {$C$} (4,2);
\draw (0,0) -- node[above] {$C$} (4,0);
\draw (0,0) -- node[below left] {$C$} (4,-2);
\end{tikzpicture}\label{diag:tree}
\]

We now prove the Claim using  Bass--Serre theory. Let $T_i$ be the Bass--Serre tree of the splitting $G_i=A_i\ast_C B$.
Thus \(G_i\backslash T_i\) is a single edge, with vertex stabilizers \(A_i\) and \(B\), and edge stabilizer \(C\).

We determine the structure of $K_i$ by considering the induced action of \(K_i\) on \(T_i\). 

First, we determine the vertices of \(K_i\backslash T_i\) of \(A_i\)-type. 
Since $A_i = \bigsqcup_{1\leq j \leq n} Hg_{ij}$, we have
$$
G_i= \pi_i^{-1} (\bigsqcup _j Hg_{ij})= \bigsqcup _j K_ig_{ij}=K_iA_i,
$$
which implies that there is one single $K_i$-orbit of $A_i$-type vertices. The stabilizer of the vertex corresponding to $A_i$ is $K_i\cap A_i=H$.

Next, we determine the vertices of $B$-type. Every $K_i$-orbit of $B$-type vertices of $T_i$ is represented by $aB$ for some $a\in A_i$, because $G_i=K_iA_i$. Two such vertices $aB$ and $a'B$ are in the same $K_i$-orbit if and only if
$$
a'\in K_i aB.
$$
Applying \(\pi_i\), and using \(\pi_i(K_i)=H\) and \(\pi_i(B)=C\), this is
equivalent to
$$
a'\in HaC .
$$
Since $C$ is normal in $A_i$ (implying $aC=Ca$) and \(C\leq H\), this is equivalent to
\[
a'\in Ha.
\]
Therefore the \(K_i\)-orbits of \(B\)-type vertices are indexed by $H \backslash A_i=\{ Hg_{i1} , Hg_{i2}, \cdots , Hg_{in} \}$. Then the quotient graph has $n$ $B$-type vertices, say $g_{i1}B,g_{i2}B,\cdots, g_{in}B$.

Let  $1\leq j\leq n$, since $g_{ij}Bg_{ij}^{-1}\leq \pi_i^{-1} (g_{ij}Cg_{ij}^{-1})=\pi_i^{-1} (C) \leq \pi_i^{-1}(H)=K_i$, the stabilizer of $g_{ij}B$ is 
$$
K_i\cap g_{ij}Bg_{ij}^{-1}=g_{ij}Bg_{ij}^{-1}.
$$ 

Similarly, one shows there are $n$ edges, $g_{i1}C,g_{i2}C,\cdots g_{in}C$, and the edge $g_{ij}C$ joins the vertex $A_i$ to the vertex 
 $g_{ij}B$  with stabilizer $g_{ij}Cg_{ij}^{-1}$.

So far we have proved that $K_i$ is the following graph of groups:

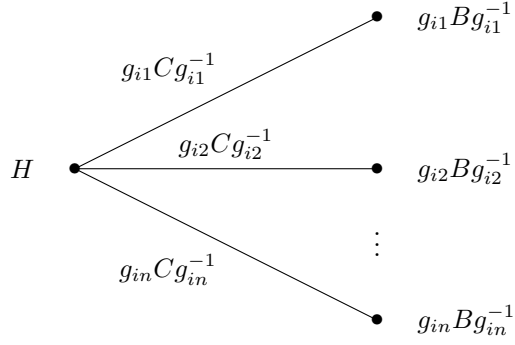
\[
\begin{tikzpicture}[baseline=(current bounding box.center)]
\node (A) at (0,0) {$\bullet$};
\node[left=0.2cm of A] {$H$};

\node (B1) at (4,2) {$\bullet$};
\node[right=0.2cm of B1] {$g_{i1}Bg_{i1}^{-1}$};

\node (Bmid) at (4,0) {$\bullet$};
\node[right=0.2cm of Bmid] {$g_{i2}Bg_{i2}^{-1}$};

\node[below=0.3cm of Bmid] {$\vdots$};

\node (Bn) at (4,-2) {$\bullet$};

\node[right=0.2cm of Bn] {$g_{in}Bg_{in}^{-1}$};

\draw (0,0) -- node[above left] {$g_{i1}Cg_{i1}^{-1}$} (4,2);
\draw (0,0) -- node[above] {$g_{i2}Cg_{i2}^{-1}$} (4,0);
\draw (0,0) -- node[below left] {$g_{in}Cg_{in}^{-1}$} (4,-2);
\end{tikzpicture}
\]
where the homomorphisms from $g_{ij}Cg_{ij}^{-1}$ to $H$ map $g_{ij}cg_{ij}^{-1}\in g_{ij}Cg_{ij}^{-1}$ to $g_{ij}cg_{ij}^{-1}$, and the homomorphisms from $g_{ij}Cg_{ij}^{-1}$ to $g_{ij}Bg_{ij}^{-1}$ are inclusions. By condition (3), all edge maps are naturally isomorphic to canonical inclusions $H\hookleftarrow C \hookrightarrow B$, which form exactly the claimed graph of groups.
\end{proof}

\begin{rem}
   Proposition \ref{lem:common-finite-index-subgroup} may fail without assumption (3).  Here is an example where the weakened hypotheses hold, but
$G_1$ and $G_2$ are not commensurable. Let
\(
C=\langle x,y \mid [x,y]=1\rangle \cong \mathbb Z^2,
H=C.
\)
Choose
\( \alpha=
\begin{pmatrix}
2 & 1\\
1 & 1
\end{pmatrix} \in \mathrm{GL}_2(\mathbb Z).\)
Let  $N\cong \mathbb Z^2$, and define
\(
B=N\rtimes C
\)
where $x$ acts on $N$ by $\alpha$, while $y$ acts trivially.
Then the projection
\(
r:B=N\rtimes C\longrightarrow C
\)
is a retraction. Define
\(
A_1=C\times (\bZ/2\bZ)
\)
and
\(
A_2=C\rtimes (\bZ/2\bZ)=C\rtimes \langle s\mid s^2=1\rangle,
\)
where $s$ swaps the two generators:
\[
sxs^{-1}=y,
\qquad
sys^{-1}=x.
\]
Set
\(
G_i=A_i\ast_C B.
\)
We claim that $G_1$ and $G_2$ are not commensurable. Suppose there are isomorphic finite-index subgroups $K'<G_1,K<G_2$ with an isomorphism $\phi:K'\rightarrow K$. 
By the proof of Proposition \ref{lem:common-finite-index-subgroup}, the finite-index subgroup $K_2<G_2$ is the semi-direct product $(N_0*N_1)\rtimes C$, where $N_0,N_1\cong \mathbb Z^2$ and $x$ acts on $N_1$ by $\alpha$, on $N_0$ trivially, while $y$ acts on $N_0$ by $\alpha$, on $N_1$ trivially. Since the element $y$ is central in $G_1$, we have an infinite-order central element $\phi(y^n)\in K \cap K_2$ for some positive integer $n$. Let $$p:K_2=(N_0*N_1)\rtimes \mathbb{Z}^2 \rightarrow (\mathbb{Z}^2 \times \mathbb{Z}^2)\rtimes \mathbb{Z}^2$$ be the epimorphism induced by the abelianization of $N_0 \ast N_1$. Considering the action of $x,y$ and the fact that $p(K\cap K_2)$ contains some powers of $x,y$, we see that the central element $p(\phi(y^n))$ is trivial. This implies that $\phi(y^n)\in ker(N_0*N_1 \rightarrow \mathbb{Z}^2 \times \mathbb{Z}^2)$,  which is impossible as the element cannot centralize the intersection $K \cap (N_0*N_1)$,  a finite-index subgroup in $N_0*N_1$. 
\end{rem}

\begin{prop}\label{prop:strong-comm-graph-sig}
    Let $\Gamma G_v$ and $\Gamma G_v^\ast$ be two graph products on the same underlying graph $\Gamma$. 
    Suppose that $G_v = G_v^\ast$ for all vertices of $\Gamma$ except $v_0$, and that $G_{v_0}$ and $ G_{v_0}^\ast$ are strongly commensurable with a common subgroup $H_{v_0}$. 
    Then the graph products $\Gamma G_v$ and $\Gamma G_v^\ast$ are also strongly commensurable with a common subgroup $K= \pi^{-1}(H_{v_0}) \cong \pi^{\ast-1}(H_{v_0})$ of index $[G_{v_0}: H_{v_0}]$, where $\pi:\Gamma G_v \to G_{v_0}$ (resp. $\pi^\ast:\Gamma G^\ast_v \to G_{v_0}^\ast$) is the projection onto the $v_0$ factor.
\end{prop}

\begin{proof}
By Proposition \ref{prop:amal-graph}, we have the following amalgamated product of the graph products:
\[
\Gamma G_v = (G_{v_0}\times G_{lk(v_0)}) \ast_{G_{{lk}(v_0)}}  \Gamma_0 G_{v} \]
\[\Gamma G_v^\ast = (G_{v_0}^\ast\times G^\ast_{lk(v_0)}) \ast_{G^\ast_{{lk}(v_0)}}  \Gamma_0 G^\ast_v
\]
By our assumption, $ \Gamma_0 G_{v}  = \Gamma_0 G^\ast_v, G_{{lk}(v_0)}= G^\ast_{{lk}(v_0)} $, and $G_{{lk}(v_0)}$ is a retract of $\Gamma_0 G_{v}$. Here $\Gamma_0$ is the full subgraph of $\Gamma$ generated by vertices of $\Gamma$ other than $v_0$. $G_{lk(v_0)}$ (resp. $G^\ast_{lk(v_0)}$) is clearly normal in $G_{v_0}\times G_{lk(v_0)}$ (resp. $G^\ast_{v_0}\times G^\ast_{lk(v_0)}$). The common index $[G_{v_0}: H_{v_0}]$ subgroup $H = H_{v_0} \times G_{lk(v_0)} = H_{v_0} \times G^\ast_{lk(v_0)}$ contains $ G_{lk(v_0)} =  G^\ast_{lk(v_0)}$. So the conditions of Proposition \ref{lem:common-finite-index-subgroup} are all satisfied; hence $\Gamma G_v$ and $\Gamma G_v^\ast$ have an index $[G_{v_0}:H_{v_0}]$ common subgroup $\pi^{-1} (H_{v_0}) \cong \pi^{\ast-1} (H_{v_0})$.
\end{proof} 

 With Proposition \ref{prop:strong-comm-graph-sig}, the strong commensurability theorem for graph products now follows from the following simple observation:
 
 \begin{lem}\label{lem:str-com-tri}
     Let $A$ and $B$ be strongly $m$-commensurable, $B$ and $C$ be strongly $n$-commensurable. Then $A$ and $C$ are strongly $d$-commensurable for some $d\leq mn$.
 \end{lem}
\begin{proof}
Let $H$ be the common index $m$ subgroup of $A$ and $B$, $K$ be the common index $n$ subgroup of $B$ and $C$. Then $H$ and $K$ are both subgroups of $B$. So $H\cap K$ as subgroups of $H$ and $K$ are  finite index subgroups of  $A,B,C$. Let us calculate the index.

\[[A:H\cap K] = [A:H][H:H\cap K] = m[H:H\cap K];\]
\[[C:H\cap K] = [C:K][K:H\cap K] = n[K:H\cap K].\]
For $B$, we have two ways to calculate the index:

\[[B:H\cap K] = [B:H][H:H\cap K] = m[H:H\cap K], \text{ and}\]
\[[B:H\cap K] = [B:K][K:H\cap K] = n[K:H\cap K].\]
This implies that $m[H:H\cap K]=n[K:H\cap K]$, so $A$ and $C$ are strongly $m[H:H\cap K]$-commensurable. It remains to show that $[H:H\cap K] \leq n$. Consider the map of cosets $H/(H\cap K) \to B/K$. One checks it is injective. Hence $[H:H\cap K] \leq n$.
\end{proof}

\begin{cor}\label{thm:strong-com-gp} \cite[Theorem 1]{JanSwi01}
   Let $\Gamma G_v$ and $\Gamma G_v^\ast$ be two graph products on the same underlying graph $\Gamma$. Suppose that for any vertex $v\in V(\Gamma)$, $G_v$ and $G_v^\ast$ are strongly commensurable with the common subgroup $H_v$. Then the graph products $\Gamma G_v$ and $\Gamma G_v^\ast$ are also strongly $d$-commensurable with $d\leq \prod_{v\in V(\Gamma)} [G_v: H_v]$.   
\end{cor}

\begin{rem}
    Strictly speaking, \cite[Theorem 1]{JanSwi01} is slightly stronger than this as the index of their common finite index subgroup is precisely  $\prod_{v\in V(\Gamma)} [G_v: H_v]$.
\end{rem}

\begin{proof}[Proof of Corollary \ref{thm:strong-com-gp}]
One can build a sequence of graph products of groups so that each consecutive one differs only at one vertex group. In fact list the vertices of $\Gamma$ by $v_1,\cdots, v_n$. Let the first graph product be $\Gamma G_v$. For the second graph product,  we replace the group $G_{v_1}$ by $G_{v_1}^\ast$. Inductively, at the last step, we get the graph product $\Gamma G_v^\ast$. By Proposition \ref{prop:strong-comm-graph-sig}, the $i$-th and $(i+1)$-th graph products in this sequence are strongly $[G_{v_i}:H_{v_i}]$-commensurable. By Lemma \ref{lem:str-com-tri} and induction, we have $\Gamma G_v$ and $\Gamma G_v^\ast$ are strongly $d$-commensurable for some $d\leq \prod_{i=1}^n [G_{v_i}:H_{v_i}]$.
\end{proof}

\begin{rem}
   In \cite[\S 4]{JanSwi01}, they actually identify the common finite index subgroup as the preimage of $\prod H_v$ under the natural projection $\Gamma G_v \to \prod_{v\in V(\Gamma)} G_v$.  Our method does not seem to be able to prove this. 
\end{rem}

\section{Graph products of RFRS groups}

In this section, we study graph products of RFRS groups.

\begin{defn}
   A finitely generated group $G$ is called \emph{RFRS} (residually finite rationally solvable) if it has a sequence of finite index  subgroups $G=G_{0}\rhd G_{1} \rhd G_{2} \rhd \cdots$
such that: 
\begin{enumerate}
    \item $\cap_i G_{i}=1$; 
    \item $G_{i+1}\geq \ker (G_{i}\rightarrow H_{1}(G_{i}; \mathbb{Q}))$ for each $i$.
\end{enumerate}

Such a tower of subgroups of $G$ is called an \emph{RFRS tower}. A \emph{partial RFRS tower} is a finite sequence of finite index  subgroups $G=G_{0}\rhd  G_{1}\rhd  \cdots \rhd  G_n$ that satisfies condition (2). 

\end{defn}

\begin{rem}\label{rem:defn-RFRS}

\begin{enumerate}
\item Condition (2) is the same as saying the quotient map $G_i \to G_{i}/G_{i+1}$ factors through a torsion-free abelian group.

\item One could further assume that each $G_i$ is normal in $G$, but as explained by Agol in \cite[Section 2]{Agol08}, this  gives the same class of groups.

\item Sometimes it might be hard to prove that a group $G$ is $\RFRS$ directly, but if we can find a partial RFRS tower $G=G_{0}\geq G_{1}\geq \cdots \geq G_n$ such that $G_n$ is $\RFRS$, then $G$ is $\RFRS$.
\end{enumerate}
    
\end{rem}

 A virtually special group
is virtually RFRS \cite[Theorem 2.2]{Agol08}.  We record the following criterion for proving $\RFRS$.

\begin{lem}
    Let $G$ be a group, $G=G_0 \rhd G_1\rhd G_2 \rhd \cdots \rhd G_n $ and $G=H_0 \rhd H_1 \rhd H_2 \rhd \cdots \rhd H_m $ be two partial RFRS towers. Then 
    \[G=G_0 \rhd G_1 \rhd \cdots \rhd G_n \rhd G_n\cap H_1 \rhd \cdots \rhd G_n\cap H_m \]
    is also a partial RFRS tower of $G$.
\end{lem}
\begin{proof}
    The only point to check is the RFRS condition (2) along the new part of the tower. 
    
    Set $K_i:=G_n\cap H_i$ for $0\leq i\leq m$, then we have the following commutative diagram.
    $$
    \begin{array}{ccc}
        K_i & \hookrightarrow  & H_i \\
        \downarrow & & \downarrow \\
        H_1(K_i; \mathbb{Q}) & \longrightarrow & H_1(H_i; \mathbb{Q})
    \end{array}
    $$
    Since any $x\in K_i$ mapped trivially to $H_1(K_i; \mathbb{Q})$ cannot survive in $H_1(H_i; \mathbb{Q})$, we have 
    $$
    \ker (K_i \to H_1(K_i; \mathbb{Q})) \leq \ker (H_i \to H_1(H_i; \mathbb{Q})) \leq H_{i+1}.
    $$
    Also $K_i \leq G_n$, hence $\ker (K_i \to H_1(K_i; \mathbb{Q})) \leq G_n \cap H_{i+1} = K_{i+1}$.
\end{proof}

\begin{cor}\label{cor:def-RFRS}
    A group $G$ is RFRS if and only if for any non-trivial element $g\in G$, there is a partial RFRS tower $G=G_0\rhd G_1\rhd \cdots \rhd G_n$ such that $g\not \in  G_n$.
\end{cor}

    Let us denote the class of groups that are both poly-$\bZ$ and virtually abelian by $\mathcal{P}_{zva}$.

\begin{prop}\label{prop:rfrs-os}
\cite[Section 6]{OkuSch24}
\begin{enumerate}
    \item Residually RFRS groups are RFRS;
    \item A group is RFRS if and only if it is residually $\mathcal{P}_{zva}$.
\end{enumerate}    
\end{prop}

\begin{lem}\label{lem:polzva-partial-RFRS}
    Let $G$ be a  group in $\mathcal{P}_{zva}$. Then there is a partial RFRS tower $G=G_{0}\geq G_{1}\geq \cdots\geq G_n$ such that $G_n$ is abelian. 
\end{lem}

\begin{proof}
    Recall that any group in the class  $\mathcal{P}_{zva}$ is RFRS. We want to make use of the following equivalent definition of RFRS in \cite[Section 6]{OkuSch24}:
    A group $G$ is RFRS  if there are two normal residual sequences of subgroups
\[
G = H_0 \geq H_1 \geq H_2 \geq \cdots
\quad \text{and} \quad
G = K_0 \geq K_1 \geq K_2 \geq \cdots
\]
such that:
\begin{enumerate}
    \item Each $H_i$ is a finite index subgroup of $G$;
    \item  $K_i \leq H_i$;
    \item For all $i \geq 0$, the quotient $H_i / K_{i+1}$ is torsion-free abelian.
\end{enumerate}

Note that since $K_i\leq H_i$, we also have $K_i/K_{i+1}$ is either trivial or torsion-free abelian. 
But since $G$ is poly-$\bZ$ in our case, and $\cap K_i =\{1\}$, we must have $K_i=\{1\}$ for all $i\geq N$ for some large $N$.  But in this case, the condition that $H_i/K_{i+1}$ is torsion-free abelian implies that $H_i$ is abelian for any $i\geq N-1$. It remains to show that $ G= H_0 \geq H_1 \geq H_2 \geq H_{N-1}$ is a partial RFRS tower. But this follows immediately from the fact that $H_i/K_{i+1}$ is torsion-free abelian and $K_{i+1}\leq H_{i+1}$.
\end{proof}


    

\begin{lem}\label{lem:polzva-rfrs}
    Graph products of  groups in $\mathcal{P}_{zva}$ are RFRS.
\end{lem}

\begin{proof}
    Let $\Gamma G_v$ be a graph product where each vertex group $G_v$ is in $\mathcal{P}_{zva}$.
    By Lemma \ref{lem:polzva-partial-RFRS}, each $G_v$ admits a partial RFRS tower of length $n_v$, that terminates in a torsion-free abelian group. Denote them by 
    $$
    G_v = G_{v,0} \rhd G_{v,1} \rhd \cdots \rhd G_{v,n_v}.
    $$

    Since $\Gamma G_v$ embeds into $\Gamma (G_v\times 
    \bZ)$ and subgroups of RFRS groups are RFRS, it suffices to show that $\Gamma (G_v\times \bZ)$ is RFRS. 
    

    
    We will reduce the RFRSness of $\Gamma (G_v\times \bZ)$ to that of a RAAG by induction. Since RAAGs are RFRS \cite[Theorem 2.2]{Agol08}, this will finish the proof. 
    Let us fix a vertex $v_0$ first. Consider the projection 
    $$\Gamma (G_v\times \bZ) \to G_{v_0}\times \bZ \to G_{v_0} \to G_{v_0}/G_{v_0,1}.$$
    Let $K_{v_0,1}$ be the kernel of the composition of the above maps. Since $G_{v_0}\rhd G_{v_0,1}$ is a partial RFRS tower,  one checks that $\Gamma (G_v\times \bZ) \rhd K_{v_0,1}$ is also a partial RFRS tower.  So it suffices to show $K_{v_0,1}$ is RFRS. Consider another graph product $\Gamma G'_{v} $, where $G_{v}' = G_v \times \bZ$ for all $v\neq v_0$, $G_{v_0}' = G_{v_0,1} \times \frac{1}{[G_{v_0}: G_{v_0,1}]}\bZ$. Applying Proposition \ref{prop:strong-comm-graph-sig} to the graph product $\Gamma (G_v\times \bZ)$ and $\Gamma G'_v$, we have $K_{v_0,1}$ is isomorphic to a subgroup of $\Gamma G_v'$. So it suffices to show that $\Gamma G_v'$ is RFRS.  Notice now  that in $\Gamma G_v'$, we have replaced the vertex group $G_{v_0} \times \bZ$ by $G_{v_0,1} \times \frac{1}{[G_{v_0}: G_{v_0,1}]} \bZ  \cong G_{v_0,1} \times \bZ$ and all other vertex groups stay the same. Inductively, using the partial RFRS tower of $G_{v_0}$, in order to prove $\Gamma (G_v\times \bZ)$ is RFRS, it suffices to prove that $\Gamma G_v^\ast $ is RFRS, where $G_v^\ast =G_{v} \times \mathbb{Z}$ for all $v\neq v_0$ and $G_{v_0}^\ast =G_{v_0,n_{v_0}} \times \mathbb{Z}$. Applying this strategy to   vertices of $\Gamma$ one by one, we see that to prove $\Gamma (G_v\times \bZ)$ is RFRS, it suffices to prove that $\Gamma (G_{v,n_v}\times \bZ)$ is RFRS. But each $G_{v,n_v}$ is now a finitely generated torsion-free abelian group, implying that  $\Gamma (G_{v,n_v}\times \bZ)$ is a RAAG. This finishes the proof.
\end{proof}

\begin{thm}\label{thm:grp-rfrs}
    Graph products of RFRS groups are RFRS.
\end{thm}

\begin{proof}
    Let $\Gamma G_v$ be a graph product where each vertex group $G_v$ is RFRS.
    
    Suppose $1 \neq g \in \Gamma G_v$ has a normal form $g= g_1 g_2\cdots g_r$. Since each $G_v$ is residually $\mathcal{P}_{zva}$, we may choose a  quotient $Q_v\in \mathcal{P}_{zva}$ for each $G_v$ such that each $g_i$ maps to a non-trivial element in its respective quotient. Since all syllables remain nontrivial and the vertex sequence is the same, Green’s normal form theorem implies that the image word is still reduced.
    
    The induced homomorphism $f: \Gamma G_v \to \Gamma Q_v$ maps g to an element in $\Gamma Q_v$ with a non-trivial normal form, thus $f(g) \neq 1$. Since $\Gamma Q_v$ is RFRS by Lemma \ref{lem:polzva-rfrs}, it follows that $\Gamma G_v$ is residually RFRS. Proposition \ref{prop:rfrs-os} now completes the proof.
\end{proof}





\section{Graph products of poly-free and normally poly-free groups}\label{section:poly-free}

We study (normal) poly-freeness of graph products in this section.

\begin{lem}\cite[Lemma 2.2]{Wu2022}
    $G$ is normally poly-free if and only if there exists a sequence of quotients
    $$
    G=Q_0 \xrightarrow{q_1} Q_1 \xrightarrow{q_2} Q_2 \xrightarrow{q_3} \cdots \xrightarrow{q_n} Q_n = \{1\}
    $$
    such that $\ker q_i$ is free for each $i$. 
\end{lem}

\begin{lem}\label{lem:graph-group-free}\cite[Lemma 2.6]{Wu2022}
    Let $C$ be a subgroup of $A$ and $B$ and $G = A \ast_C B$ be the amalgamated product. Suppose $f : G \to Q$ is a map such that $f|_C$ is injective. If $\ker(f|_A)$ and $\ker(f|_B)$ are free groups (resp. poly-free groups), then $\ker(f)$ is also a free group (resp. poly-free group).
\end{lem}

The idea of the proof for the following lemma goes back to \cite[Section 3]{Wu2022}.

\begin{lem}\label{lem:poly-free-graph-prod}
    Let $\Gamma G_v$ be a graph product of groups. Then we have a sequence of surjective maps
    \[ \Gamma G_v \xrightarrow{q_1} Q_1 \xrightarrow[]{q_2} Q_2 \cdots \xrightarrow[]{q_m} Q_m=\prod G_v, \]
     where $\ker(q_i)$ is a free group (of possibly infinite rank) for each $i$, $m= \frac{|V(\Gamma)|(|V(\Gamma)|-1)}{2} - |E(\Gamma)|$.
\end{lem}

\begin{proof}
If $\Gamma$ is a complete graph, we are done. If not, set 

\[\Gamma=\Gamma_0\subseteq \Gamma_1\subseteq\Gamma_2\cdots \cdots\subseteq \Gamma_m,\]
where each $\Gamma_{i+1}$ is obtained from $\Gamma_{i}$ by adding one edge. Note that $m=\frac{|V(\Gamma)|(|V(\Gamma)|-1)}{2} - |E(\Gamma)|$. In terms of graph products, this gives a sequence of quotient maps:

\[ \Gamma G_v \xrightarrow{q_1} \Gamma_1 G_v \xrightarrow[]{q_2} \Gamma_2 G_v \cdots \xrightarrow[]{q_m} \Gamma_m G_v=\prod G_v. \]


Let $v_i$ and $w_i$ be the two vertices of the edge newly added to $\Gamma_i$. By Proposition  \ref{prop:amal-graph}, we have an amalgamated product of $\Gamma_i G_v$:

\[  \Gamma_i G_v = G_{st(v_i)}\ast_{G_{lk(v_i)}} G_{\Gamma_i \backslash \{ v_i \}}.\]

Notice that $\Gamma_i \backslash \{ v_i \}$ is identical with $\Gamma_{i+1} \backslash \{ v_i \}$, and that the star (resp. link) of $v_i$ in $\Gamma_i$ embeds as a full subgraph in the star (resp. link) of $v_i$ in $\Gamma_{i+1}$.
Hence $q_i:\Gamma_i G_v \to \Gamma_{i+1} G_v$ restricts to the identity on $G_{\Gamma_i \backslash \{ v_i \}}$ and to injections on $G_{st(v_i)}$ and $G_{lk(v_i)}$. Now, by Lemma \ref{lem:graph-group-free}, $\ker(q_i)$ is free. 
\end{proof}

We immediately have the following criterion about the inheritance properties of graph products. Note that direct products can be obtained from extensions.

\begin{lem}\label{lem:inh-prop}
    Let $\cP$ be a group-theoretic property such that
    \begin{enumerate}
        \item $\cP$ is closed under taking extensions, and
        \item free groups  of arbitrary rank satisfy $\cP$.
    \end{enumerate}
    Then $\cP$ is closed under taking graph products.
\end{lem}

\begin{cor}\label{cor:poly-free}
   The following properties are closed under taking graph products:
\begin{enumerate}
    \item (Virtually) torsion-free;
    \item (Virtually) left-orderable; 
    \item (Virtually) locally indicable;
    \item (Virtually) poly-free;
    \item (Virtually) normally poly-free;
    \item Farrell--Jones Conjecture with finite wreath products and coefficients in additive categories.
\end{enumerate} 
\end{cor}
\begin{proof}
The first four properties without ‘‘virtual’’ satisfy the assumptions of Lemma \ref{lem:inh-prop}, so they are closed under taking graph products. By Lemma \ref{lem:poly-free-graph-prod}, the virtual case follows from this and the observation that direct products of virtually $\cP$ groups are virtually $\cP$, where $\cP$ is one of the (without ‘‘virtual’’) properties in (1)-(4).

\begin{enumerate}[start=5]
    
    


    \item Lemma \ref{lem:poly-free-graph-prod} reduces the proof to direct products. But normally poly-free groups are closed under taking direct products; the case of virtually normally poly-free groups again follows from the observation that direct products of virtually normally poly-free groups  are virtually normally poly-free.


    \item By \cite[Theorem 1.1(4)]{BrKiWu21}, in order to prove $\Gamma G_v$ satisfies the Conjecture, it suffices to prove that $Q_1$ and all free-by-cyclic groups satisfy the conjecture. By \cite[Theorem A]{BrKiWu21}, it is enough to prove that $Q_1$ satisfies the conjecture. Inductively, it suffices to show that $Q_m=\prod G_v$ satisfies the conjecture. But since each $G_v$ satisfies the conjecture and the conjecture is closed under finite direct products \cite[Theorem 1.1(4)]{BrKiWu21}, $\prod G_v$ also satisfies the conjecture.
\end{enumerate}
\end{proof}

\section{Virtual properties of groups under graph products}

The following criterion allows us to establish many virtual properties are closed under taking graph products.

\begin{lem}\label{lem:vir-inh-prop}
    Let $\cP$ be a group-theoretic property such that
    \begin{enumerate}
        \item $\cP$ passes to subgroups, i.e., if a group has $\cP$, any finitely generated subgroup of it also has $\cP$, and
        \item $\bZ$ satisfies $\cP$.
    \end{enumerate}
    If $\cP$ is preserved by graph products, then so is virtual $\cP$.
\end{lem}
\begin{proof}
    Let $\Gamma G_v$ be a graph product whose vertex groups $G_v$ are virtually $\cP$. Then there exists an index $d_v$ subgroup $H_v \leq G_v$ having $\cP$.
    By our assumptions, $H_v \times \bZ$, and hence $\Gamma (H_v \times \bZ)$, has $\cP$. Now consider the graph product $\Gamma (H_v \times \frac{1}{d_v}\bZ)$. Since $\Gamma (H_v \times \frac{1}{d_v}\bZ)$ is isomorphic to $\Gamma (H_v \times \bZ)$, it also has $\cP$. Consider the graph products $\Gamma (H_v \times \frac{1}{d_v}\bZ)$ and $\Gamma (G_v\times \bZ)$. At each vertex $v$, both vertex groups contain $H_v \times \bZ$ as an index $d_v$ subgroup.  By  Corollary \ref{thm:strong-com-gp}, $\Gamma (H_v \times \frac{1}{d_v}\bZ)$ is strongly commensurable to $\Gamma(G_v \times \bZ)$. Therefore  $\Gamma(G_v \times \bZ)$ is virtually $\cP$. This further implies that $\Gamma G_v$ is virtually $\cP$ because $\cP$ passes to finitely generated subgroups.
\end{proof}

\begin{cor}\label{cor:graph-prod-pass-sub}
    The following properties are closed under taking graph products:
\begin{enumerate}

    \item Virtually RFRS;
    \item Virtually special;
    \item Virtually orderable;
    \item Virtually poly-free;
    \item Virtually normally poly-free;
    \item Virtually locally indicable.
    
\end{enumerate} 
\end{cor}

\begin{proof}
    One checks that these properties (without ‘‘virtually’’) indeed pass to subgroups and $\bZ$ has these properties. The results now follow from Lemma \ref{lem:vir-inh-prop} and Theorem \ref{thm:graph-prod-sumry} and Corollary \ref{cor:poly-free}(5) for normally poly-free. 
\end{proof}

Not all group theory properties are closed under taking subgroups, for example being CAT(0), CAT(0) cube or compact special. But we can still deal with some of these.

\begin{thm}\label{thm:graph-product}

    The following properties are closed under taking graph products:
    
\begin{enumerate}
    \item Virtually CAT(0);
    \item Virtually CAT(0) cube;
    \item Virtually compact special.
  \end{enumerate}

\end{thm}

\begin{proof}
Let $\Gamma G_v$ be the graph product, and for each vertex $v$, $H_v$ is an index $d_v < \infty$ subgroup of $G_v$. Let $G_v^\ast = H_v\times \bZ/d_v\bZ$. By Corollary \ref{thm:strong-com-gp}, $\Gamma G_v$ and $\Gamma G_v^\ast$ are strongly commensurable. So it suffices to prove that the graph product $\Gamma G_v^\ast$ has the required properties.
\begin{enumerate}
\item Suppose $H_v$ acts properly and cocompactly via isometries on the CAT(0) space $X_v$, then $G_v^\ast=H_v\times \bZ/d_v\bZ$ also acts on $X_v$ by projecting $H_v\times \bZ/d_v\bZ$ onto $H_v$. This action is again proper and cocompact, so $G_v^\ast$ is CAT(0). The result now follows from Theorem \ref{thm:graph-prod-sumry} (13).

\item The proof is identical to the virtually CAT(0) case.

\item Let $H_v$ be the fundamental group of a compact special cube complex $X_v$. Then $H_v\times \bZ/d_v\bZ$  acts geometrically on the universal cover $\tilde{X}_v$ through projection onto the first factor. By (2), the graph product $\Gamma G_v^\ast$ also acts geometrically on some CAT(0) cube complex. By \cite[Theorem 8.17]{genevois2017cubical}, the graph product $\Gamma G_v^\ast$ acts geometrically and virtually specially (in the sense of \cite[Definition  4.4]{genevois2017cubical}) on some CAT(0) cube complex, say $X$. On the other hand, we also have $\Gamma G_v^\ast$ is virtually special by Corollary \ref{cor:graph-prod-pass-sub}. In particular, $\Gamma G_v^\ast$ is residually finite and virtually torsion-free. Thus $\Gamma G_v^\ast$ has a finite index  torsion-free subgroup $H$ acting freely on $X$. In this case, acting specially in the sense of \cite[Definition  4.4]{genevois2017cubical} implies acting specially in the sense of  \cite[Definition 9.11]{Hagwise08}; see also \cite[Remark 4.5 \& 4.6]{genevois2017cubical}. Now $H\backslash X$ clearly is compact since $H$ is a finite index subgroup of $\Gamma G_v^\ast$ and $\Gamma G_v^\ast \backslash X$ is compact. By  \cite[Propsition 3.10]{Hagwise08}, up to taking a further finite index subgroup, we can assume that $H\backslash X$ is $A$-special in the sense of Haglund and Wise \cite{Hagwise08}. Thus, $\Gamma G_v^\ast$ is virtually compact special.

\end{enumerate} 
\end{proof}

\section{Stable word length}

In this section, we show that the class of groups with positive (resp. discrete) stable word length is closed under taking graph products. Inspired by Proposition \ref{lem:common-finite-index-subgroup}, we will actually deal with a general notion that we introduce below.

\begin{defn}
A group $G$ is called inductively  amalgamated if there exist subgroups $%
B_{0}=C_{0},C_1\leq B_{1},...,C_n\leq B_{n},$ such that 
\begin{eqnarray*}
G_{0} &=&B_{0},G_{1}=G_{0}\ast _{C_{1}}B_{1},\cdots ,G_{i}=G_{i-1}\ast
_{C_{i}}B_{i}, \\
\cdots ,G &=&G_{n}= G_{n-1}\ast _{C_{n}}B_{n}.
\end{eqnarray*}%
In this case, we say that $G$ is inductively  amalgamated by $\{C_{i}\leq B_{i}\}.$ If
each $C_{i}$ is normal in $G_{i-1}$ (resp. $B_{i}$) we call the inductively  amalgamated product  left (resp. right) normal. If each $C_{i}$ is a
retract of $G_{i-1}$ (resp. $B_{i}$) we call the inductively  amalgamated product
 left (resp. right) retracting.
\end{defn}

\begin{ex} \label{examp-graph-product}
\begin{enumerate}
\item A graph product $\Gamma G_{v}$ has a left-right-retracting and right
normal inductive amalgamation.

\item Let $\Gamma $ be a finite tree labeled by positive integers. The Artin
group $A(\Gamma )$ is a left-right-retracting inductive amalgamation.

\item Let $\Gamma $ be a finite simplicial graph labeled by even positive integers. The Artin
group $A(\Gamma )$ has a left-right-retracting inductive amalgamation (by the clique subgroups).

\end{enumerate}
\end{ex}

\begin{proof}
Let $\Gamma $ be a finite graph with vertices $\{v_{j}\}_{j=0}^{n}$ and $%
\Gamma _{i}$ be the subgraph spanned by $\{v_{j}\}_{j=0}^{i}.$ Choose $%
B_{0}=G_{v_{0}}.\ $ If $v_{0},v_{1}$ are adjacent, choose $%
B_{1}=G_{v_{0}}\times G_{v_{1}}$ and $C_{1}=B_{0}.$ Otherwise, choose $%
B_{1}=G_{v_{1}}$ and $C_{1}=\{e\}.$ Inductively, the group $G_{i}$ is the
graph product spanned by the subgraph $\Gamma _{i}$ on $\{v_{j}\}_{j=0}^{i}.$
Let $lk(v_{i+1})$ be the link graph of $v_{i+1}.$ Choose $C_{i+1}$ to be the
graph product induced by $\Gamma _{i}\cap lk(v_{i+1})$ and $%
B_{i+1}=G_{v_{i+1}}\times C_{i+1}.$ Note that $C_{i+1}$ is a retract of $%
G_{i}$ and the amalgamated product $G_{i}\ast _{C_{i+1}}B_{i+1}=G_{i+1}$ is
isomorphic to the graph product induced by $\Gamma _{i+1}.$ A repeated
argument finishes the proof of (1).

Note that each dihedral Artin group has a retract onto its vertex groups.
In the even labeled Artin group,  the star subgroup of a vertex retracts onto the link subgroup by mapping the vertex generator trivially.  The cases (2) and (3) can now be proved similarly.
\end{proof}

Here is an interesting observation. 

\begin{lem}
If $G$ is right-retracting inductively  amalgamated by subgroups $%
C_{i}\leq B_{i},0\leq i\leq n$. Suppose $B_0$ and $ker(B_{i}\rightarrow C_i)$ are poly-free for all $i\geq 1$, then $G$ is poly-free.
\end{lem}

\begin{proof}
We prove by induction on $i$. Let $p_{i}:B_{i}\rightarrow C_{i}$ be the retraction and $\pi
_{i}:G_{i}\rightarrow G_{i-1}$ the retraction defined by mapping $G_{i-1}$
identically and $B_{i}$ onto $C_{i}$. Lemma \ref{lem:graph-group-free} now implies that $%
\ker \pi _{i}$ is a free product of a free group and conjugates of $\ker
p_{i}.$ Therefore, $\ker \pi _{i}$ is poly-free \cite[Lemma 2.5]{Wu2022}. This implies if $G_{i-1}$ is poly-free, so is $G_i$. The lemma now follows from induction.
\end{proof}

Theorem \ref{thm:intro-virtual} (7) is a special case of the following result, considering Example \ref{examp-graph-product}.

\begin{lem}\label{lem:swl}
Let $G$ be left-right-retracting inductively  amalgamated by subgroups $%
C_{i}\leq B_{i},i=1,2,...,n.$ If each $B_{i}$ has  positive (resp. discrete)
stable word length, so does $G$.
\end{lem}

\begin{proof}
The proof is done by induction on $i$. Suppose that $G_{i-1}$ and $B_{i}$ have positive stable word length. Fix a
generating set $S$ for $G_{i}=G_{i-1}\ast _{C_{i}}B_{i}.$ Since $C_{i}$
is a retract of $G_{i-1}$ and $B_{i},$ we have the induced epimorphism $\pi
_{1}:G_{i}\rightarrow G_{i-1},\pi _{2}:G_{i}\rightarrow B_{i}.$ Let $g\in
G_{i}$ be an infinite-order element. If $g$ is elliptic, some conjugate of $g$ lies in $G_{i-1}$ (or $%
B_{i}$), write $g=s_{1}s_{2}...s_{n}$ for $n=|g|_{S}$ and $s_{i}\in S.$ Then 
\begin{equation*}
\pi _{1}(g)=\pi _{1}(s_{1})\pi _{1}(s_{2})...\pi _{1}(s_{n}),
\end{equation*}%
implying 
\begin{equation*}
\lim_{k\rightarrow \infty }\frac{|g^k|_{S}}{k}\geq \lim_{k\rightarrow \infty }%
\frac{|\pi _{1}(g^k)|_{\pi _{1}(S)}}{k}>0.
\end{equation*}%
If $g$ is hyperbolic, write $g^{k}=s_{1}s_{2}...s_{n}$ for $n=|g^{k}|_{S}$
and $s_{i}\in S.$ Choose a fixed vertex $x$ in the Bass-Serre tree $T$, we have:
\begin{eqnarray*}
d(g^{k}x,x) &=&d(s_{1}s_{2}...s_{n}x,x)\leq
d(s_{1}s_{2}...s_{n}x,s_{1}s_{2}...s_{n-1}x)+d(s_{1}s_{2}...s_{n-1}x,x) \\
&= &d(s_{n}x,x)+d(s_{1}s_{2}...s_{n-1}x,x) \\
&\leq &n\max \{d(sx,x):s\in S\}.
\end{eqnarray*}%
This implies that%
\begin{equation*}
\lim_{k\rightarrow \infty }\frac{|g^{k}|_{S}}{k}\geq \frac{1}{\max
\{d(sx,x):s\in S\}}\lim_{k\rightarrow \infty }\frac{d(g^{k}x,x)}{k}%
\geq \frac{1}{\max
\{d(sx,x):s\in S\}},
\end{equation*}%
since $\lim_{k\rightarrow \infty }\frac{d(g^{k}x,x)}{k}=\inf_{y\in T} d(gy,y) \geq 1$.
The discreteness holds when the stable word lengths on $G_{i-1}$ and $B_i$ are discrete.
\end{proof}

\bibliographystyle{alpha}
\bibliography{references.bib}

\end{document}